\documentclass{article}%
\usepackage{hyperref}
\usepackage{amsfonts}
\usepackage{amsmath}
\usepackage{amssymb}
\usepackage{amsthm}%
\usepackage{graphicx}
\providecommand{\U}[1]{\protect\rule{.1in}{.1in}}
\newtheorem{thm}{Theorem}[section]
\newtheorem{definition}[thm]{Definition}
\newtheorem{lem}[thm]{Lemma}
\newtheorem{cor}[thm]{Corollary}
\newtheorem{prop}[thm]{Proposition}
\theoremstyle{definition}
\newtheorem{exa}{Example}
\newtheorem*{rem}{Remark}
\newtheorem*{oldproof}{Proof}
\renewenvironment{proof}[1][{}]{\begin{oldproof}[#1]}{\qed\end{oldproof}}
\begin{document}

\title{Submodular functions in additive combinatorics problems for group actions and representations}
\author{Vincent Beck, Univ. Orl\'eans, CNRS IDP, UMR 7013,\\Route de Chartres, Orl\'eans F-45000, France \\C\'{e}dric Lecouvey, Univ. Tours, CNRS IDP, UMR 7013, \\Campus de Grandmont, Tours F-37000, France}
\maketitle

\begin{abstract}
This article deals with generalisations of some classical problems and results
in additive combinatorics of groups to the context of group actions or group
representations. We show that the classical methods are sufficiently deep to
extend to this wider context where, instead of two free transitive commuting
actions (left and right multiplications on the group), there is only one
single action. Following ideas of Hamidoune and Tao, our main tool is the
notion of $G$-invariant submodular function defined on power sets. We are
able to extend to this group action context results of Hamidoune and Tao as
well as results of Murphy and Ruzsa.

\end{abstract}

%\noindent AMS classification: 05E16, 12E15, 11P70.

\section{Introduction}

Consider a multiplicative group $G$ acting on the left on a non empty set $X$.
When $A$ and $Y$ are respectively finite nonempty subsets of $G$ and $X$ what
can be said about the cardinality $\left\vert A\cdot Y\right\vert $ of the set
$A\cdot Y=\{a\cdot y\mid(a,y)\in A\times Y\}$ ? Here $a\cdot y$ means the
image of $y$ under the action of $a$. When $X=G$ and the action considered is
the action by left multiplication (thus $a\cdot y=ay$, the product of the two
elements in the group $G$), this question relates to additive (or here
multiplicative) combinatorics on groups and there exist in the literature
numerous results yielding lower and upper bounds for the cardinality of the
Kronecker product set $AY$ (see for example \cite{Nath} and \cite{Tao}). Among
them, Kneser's theorem is a corner stone claiming that in any Abelian group%
\[
\left\vert G_{AY}\right\vert +\left\vert AY\right\vert \geq\left\vert
A\right\vert +\left\vert Y\right\vert
\]
where $G_{AY}=\{g\in G\mid gAY=AY\}$ is the stabilizer of the product set $AY
$. This theorem does not remain true for non Abelian groups even it is not
immediate to find a simple counterexample. Therefore, if we consider
$G_{A\cdot Y}=\{g\in G\mid g\cdot(A\cdot Y)=A\cdot Y\}$, the inequality%
\begin{equation}
\left\vert G_{A\cdot Y}\right\vert +\left\vert A\cdot Y\right\vert
\geq\left\vert A\right\vert +\left\vert Y\right\vert \label{Kneser}%
\end{equation}
\emph{does not hold} in the general left action context.\ In contrast, it is
very easy to find a counterexample by considering the action of the symmetric
group $\mathfrak{S}_{n}$ on the set $\{1,\ldots,n\}$ (see Example
\ref{Example_CounterKneser}).

Although Kneser's theorem does not have an immediate generalisation in the
group action context, we shall see in this paper that it is nevertheless
possible to obtain interesting analogues of various other results in this
setting, most of them being inspired by results or tools coming from additive
combinatorics for non Abelian groups. Among them is the notion of submodular
function defined on subsets of $G$ or subsets of $X$.\ In fact, we will often
obtain two different families of statements by fixing $Y$ and letting $A$
running on $\mathcal{P}(G)$ (the power set of $G$) or fixing $A$ and letting
$Y$ running on $\mathcal{P}(X)$ (the power set of $X$). This is for example
the case for Theorems \ref{Th_Tao} and \ref{Th_Taod} which both are
declinations of the same theorem proved by Tao in \cite{Tao} for product sets
in general groups. Even if the group action context studied in this paper
presents some analogies with the combinatorics of groups (i.e.\! the case of
an action by multiplication), it is worth mentioning that there are important
differences. Maybe the most important comes from the fact that the
multiplication in a group can be performed on the left and on the right and
that it corresponds to the case of two free commuting actions on $G$ whereas a
group action on $X$ is only one-hand sided. This makes many classical tools
like the Dyson or Diderrich transforms on subsets of groups (see for example
\cite{Die} and \cite{Nath}) irrelevant for group actions.\ 

The present paper can also be regarded as a contribution to the general
project aiming at extending methods developed in additive combinatorics of
groups to more general contexts and, as such, it has been thought to be as
self-contained as possible.\ In the linear context, where the cardinalities of
sets are replaced by the dimensions of vectors spaces, this was initiated in
\cite{Hou0} for field extensions and developed in particular in \cite{BCZ},
\cite{EL2}, \cite{Lec} (for fields and division rings) and in \cite{BL} and
\cite{MZ} (for associative algebras). As far as we are aware, the group action
setting presented in this paper was first considered very more recently in
\cite{Mur} and \cite{Murp2} in connection with the notion of approximate
groups. Our approach here, based on tools coming from group theory and on the
notion of submodular function, is different. Most often, we are also able to
state linear analogues of our results where group actions on finite sets are
replaced by finite-dimensional group representations.

\bigskip

Let us now describe more precisely the content of the paper.
Section~\ref{sec-action} is devoted to the presentation of the context of the
article: group actions and representations. Section~\ref{sec-classical}
presents methods and examples that extend positively or negatively to the
context of group actions. Our aim is also to show that not every result can be
generalised to the group action context. In particular, we explain how the
problem of determining lower and upper bounds for the previous cardinality
$\left\vert A\cdot Y\right\vert $ can theoretically be reduced to the
classical group setting when sufficient information on the orbit decomposition
and the stabilizers of the elements is available.\ This is for example the
case for free actions. Nevertheless in general, this reduction is not easy to
perform and the results are not so simple and elegant as in the group setting.
We also consider the particular case of a faithful action which gives 
straightforward counterexamples to Kneser's inequality~(\ref{Kneser}). For more
positive results, we establish results in the spirit of the paper \cite{Mur}
by Murphy and we give an analogue of a theorem by Ruzsa~\cite{Ruz} for the
action of a product set $AB$ in the group $G$ on a subset $Y$ of $X$ to
illustrate that many other classical results in additive combinatorics
certainly have interesting counterparts in the group action context. The
further sections are devoted to the use of submodular functions to generalise
theorems of classical additive group theory to the context of group actions
and group representations. Section~\ref{sec-submodular} presents the notion of
submodular map and give the standard examples that will be studied in the
following. We define in particular a natural analogue of the classical graph
cut submodular function (see Proposition~\ref{Prop_Submod_cut}).
Section~\ref{sec-fragment-atom} develops the theory of fragment and atom of
Hamidoune (see \cite{Hami2} and \cite{Tao3}) in the context of group action
and representations. In particular, Proposition \ref{Prop_LMF} gives some
information on the structure of the atoms associated to a $G$-invariant
submodular function defined on $\mathcal{P}(X)$.\ In
Section~\ref{sec-application}, we state and prove the analogues of theorems by
Hamidoune, Tao and Petridis in our group action and group representation
setting which are at the heart of this paper. These analogues rely on the
submodularity of the maps introduced in Section~\ref{sec-submodular}. We also
study in details the fragment for one of these maps. Finally, in
Section~\ref{sec-generelisation}, we end our article with another extension of
a classical result whose proof needs the notion of submodular map on a lattice.

\section{Group actions and representations context}

\label{sec-action}

%\subsection{Group actions and representations}

In the sequel we consider $G$ a group and $X$ a set on which $G$ acts. As
usual, for any $(g,x)\in G\times X$, we shall denote by $g\cdot x$ the element
of $X$ corresponding to the action of $g$ on $x$. Let us write%
\[
G_{x}=\{g\in G\mid g\cdot x=x\}
\]
for the stabilizer of $x$ in $G$. For any subset $Y\subset X$ and any $g\in G
$, set $g\cdot Y=\{g\cdot y\mid y\in Y\}$. Let%
\[
G_{Y}=\{g\in G\mid g\cdot Y=Y\}
\]
be the stabilizer of $Y$ in $G$. Observe that for any fixed $g\in G$, the map%
\begin{equation}
\left\{
\begin{array}
[c]{r@{\,}c@{\,}l}%
X & \longrightarrow & X\\
x & \longmapsto & g\cdot x
\end{array}
\right. \label{bij}%
\end{equation}
is bijective. In particular, for any finite subset $Y\subset X$, we have
$\left\vert g\cdot Y\right\vert =\left\vert Y\right\vert ,$ that is the sets
$g\cdot Y$ and $Y$ have the same cardinality. It also shows that a group
action of $G$ on a set $X$ may be given by a group homomorphism from $G$ to
the group $\mathfrak{S}(X)$ of permutations of the set $X$. The action of $G$
on $X$ is said to be faithful when the corresponding homomorphism from $G$ to
$\mathfrak{S}(X)$ is injective.

For any subset $A\subset G$ and $Y\subset X$, define%
\[
A\cdot Y=\{a\cdot y\mid(a,y)\in A\times Y\}.
\]
In the sequel, we will study lower and upper bounds for the cardinality
$\left\vert A\cdot Y\right\vert $ when $A$ and $Y$ are supposed to be
finite.\ In the particular case $X=G$ and $G$ acts on itself by left
translation, we recover the classical problem in additive combinatorics of
determining lower and upper bounds for Minkowski products of finite subsets of
an ambient group.

It will also be interesting to replace the set $X$ by its linear analogue,
that is, to consider a representation $(\rho,V)$ of the group $G$ instead of
an action of $G$ on $X$. Recall that a representation $(\rho,V)$ is a group
homomorphism $\rho:G\rightarrow GL(V)$ where $V$ is a finite-dimensional
vector space over a given field $k$.\ This can essentially be thought as a
linear action of $G$ on the vector space $V$ and we will write $g\cdot v$ the
action of any element $g\in G$ on any vector $v\in V$. We thus have for any
$(\lambda_{1},\lambda_{2})\in k^{2}$ and any $(v_{1},v_{2})\in V^{2}$%
\[
g\cdot(\lambda_{1}v_{1}+\lambda_{2}v_{2})=\lambda_{1}(g\cdot v_{1}%
)+\lambda_{2}(g\cdot v_{2}).
\]
For any subset $Z$ in $V$, we denote by $\langle Z\rangle$ the $k$-subspace of
$V$ generated by the vectors in $Z$. We then write for short $\dim(Z)$ instead
of $\dim(\langle Z\rangle)$. Given any $k$-subspace $W$ of $V$ and any subset
$A$ of $G$, we will study the dimension $\dim(A\cdot W)$ of the set%
\[
A\cdot W=\langle a\cdot v\mid(a,v)\in A\times W\rangle.
\]
in terms of $\dim(W)$ and $\left\vert A\right\vert $.

\section{Extensions and limits of standard techniques}

\label{sec-classical}

This section is devoted to the continuation of Murphy's work ~\cite{Mur} on
some extensions of classical results in combinatorial group theory to the
action group setting. The notion of symmetric sets introduced by Murphy allows
us to generalise a theorem of Fre\u{\i}man~\cite{freiman}
(Subsection~\ref{subsection-murphy}). In Subsection~\ref{subsection-ruzsa}, we
show that the proof of Theorem 9.2 of~\cite{Ruz} can be extended to the group
action context. But we start the section with some obstructions: we show that
the most classical method in the study of a group action, namely the orbit
decomposition method, is not so powerful in our combinatorial context because
it requires much information on the stabilizers and the associated left
cosets. We also exhibit counterexamples to the direct generalisation of
Kneser's theorem in the group action setting.

\subsection{Orbit decomposition method for a group action}

\label{subsection-orbit}

In this paragraph, we will assume that the set $X$ is finite. Given an element
$x$ in $X$, we denote by $O_{x}=\{g\cdot x\mid g\in G\}$ its orbit
%and by $G_{x}=\{g\in G\mid g\cdot x=x\}$ its stabilizer
.\ Let us fix $x_{1},\ldots,x_{r}$ in $X$ so that
\[
X={\textstyle\bigsqcup\limits_{i=1}^{r}}O_{x_{i}}%
\]
is the disjoint union of the orbits $O_{x_{i}},i=1,\ldots,r$.\ It is classical
that for any $i=1,\ldots,r$ the map%
\[
\phi_{i}:\left\{
\begin{array}
[c]{c}%
G/G_{x_{i}}\rightarrow O_{x_{i}}\\
gG_{x_{i}}\longmapsto g\cdot x_{i}%
\end{array}
\right.
\]
is well-defined and bijective. Assume now that we have fixed a representative
$g[i]$ in each left coset $gG_{x_{i}}$ of $G/G_{x_{i}}.$ Also for any subset
$S$ in $G$, write $\left\vert S\right\vert _{i}$ for the cardinality of the
set of cosets $\phi_{i}(S)=\{gG_{x_{i}}\mid g\in S\}$ which is the same as the
cardinality of the set $S[i]=\{g[i]\mid\phi_{i}(g[i])\in\phi_{i}(S)\}$.\ 

For any subset $Y\subset X$, write $Y_{i}=Y\cap O_{x_{i}}$. Then, we have for
any subset $A\subset G$%
\[
A\cdot Y={\textstyle\bigsqcup\limits_{i=1}^{r}}
A\cdot Y_{i}.
\]
Finally, for any $i=1,\ldots,r$, we get by setting $B_{i}=\{g[i]\mid
gG_{x_{i}}\in\phi_{i}^{-1}(Y_{i})\}$ the equalities $\left\vert Y_{i}%
\right\vert =\left\vert B_{i}\right\vert ,$ $\left\vert A\cdot Y_{i}%
\right\vert =\left\vert AB_{i}\right\vert _{i}$ and
\[
\left\vert A\cdot Y\right\vert =\sum_{i=1}^{r}\left\vert AB_{i}\right\vert
_{i}.
\]
Therefore, the problem of studying the cardinality of $\left\vert A\cdot
Y\right\vert $ can be formally reduced to the problem of studying first each
product set $AB_{i}$ in the group $G$ and next the number $\left\vert
AB_{i}\right\vert _{i}$ of left cosets attained by the elements of $AB_{i}%
$.\ Since we have
\[
\frac{\left\vert AB_{i}\right\vert }{\left\vert G_{x_{i}}\right\vert }%
\leq\left\vert AB_{i}\right\vert _{i}\leq\left\vert AB_{i}\right\vert
,i=1,\ldots,r
\]
we get%
\[
\sum_{i=1}^{r}\frac{\left\vert AB_{i}\right\vert }{\left\vert G_{x_{i}%
}\right\vert }\leq\left\vert A\cdot Y\right\vert \leq\sum_{i=1}^{r}\left\vert
AB_{i}\right\vert
\]
which theoretically reduces the question to classical estimations of product
sets in groups which is largely addressed in the literature. In particular,
when the action is simply transitive (that is when there is only one orbit and
each stabilizer is trivial), both problems are equivalent. When the action is
free (each stabilizer is trivial) we just get%
\[
\left\vert A\cdot Y\right\vert =\sum_{i=1}^{r}\left\vert AB_{i}\right\vert
\]
so that the study of $\left\vert A\cdot Y\right\vert $ can be initialized by
determining the orbits of the action of $G$ on $X$. However, in the general
case, in addition to the orbit decomposition, this method requires much
information on the different stabilizers, their associated left cosets and the
maps $\phi_{i},i=1,\ldots,r$.\ This makes the results not so simple and
elegant as in the group setting. So other methods are needed. The last two
subsections of this section show how some standard methods of combinatorial group theory can
be adaptated. But, we first study generalisations of Kneser's theorem in the group action context.

\subsection{Counterexample to Kneser's theorem for group action}

In this subsection, we exhibit two counterexamples that show that Kneser's
theorem cannot be directly extended to the group action context.

\begin{exa}
\label{Example_CounterKneser}Assume $Y=\{1,\ldots,k\}$, $G=\mathfrak{S}_{n}$
and consider $n>\ell\geq k$.\ Let $A_{0}$ be the set of permutations
$\sigma\in\mathfrak{S}_{n}$ such that $\sigma(\{1,\ldots,k\})\subset
\{1,\ldots,\ell\}$. One easily checks that
\[
\left\vert A_{0}\right\vert =\frac{\ell!}{(\ell-k)!}(n-k)!.
\]
Now, observe that $A_{0}\cdot Y=\{1,\ldots,\ell\}$ and thus the stabilizer
$G_{A_{0}\cdot Y}$ of $A_{0}\cdot Y$ has cardinality
\[
\left\vert {G}_{A_{0}\cdot Y}\right\vert =\ell!\times(n-\ell)!.
\]
So we get
\[
\left\vert A_{0}\cdot Y\right\vert +\left\vert {G}_{A_{0}\cdot Y}\right\vert
\geq\left\vert Y\right\vert +\left\vert A_{0}\right\vert \Longleftrightarrow
\ell+\ell!\times(n-\ell)!\geq k+\frac{\ell!}{(\ell-k)!}(n-k)!
\]
Hence
\[
\left\vert A_{0}\cdot Y\right\vert +\left\vert {G}_{A_{0}\cdot Y}\right\vert
\geq\left\vert Y\right\vert +\left\vert A_{0}\right\vert \Longleftrightarrow
\frac{\ell-k}{\ell!(n-\ell)!}\geq\binom{n-k}{n-\ell}-1
\]

%Version de l'inégalité précédente avec G_{A_0} Y à la place de Y et G_{A_0} A_0 à la place de A_0.
%Moreover $G_{A_{0} \cdot Y} Y =\{1,\ldots, \ell\}$ and $G_{A_{0} \cdot Y} A_0 = A_0$. We so get%
%\[
%\left\vert A_{0}\cdot Y\right\vert +\left\vert {G}_{A_{0}\cdot
%Y}\right\vert \geq\left\vert G_{A_{0} \cdot Y} Y\right\vert +\left\vert G_{A_{0} \cdot Y} A_{0}\right\vert
%\Longleftrightarrow\ell+\ell!\times(n-\ell)!\geq \ell+\frac{\ell!}{(\ell
%-k)!}(n-k)!\Longleftrightarrow 1 \geq\binom
%{n-k}{n-\ell}
%%\frac{(n-\ell+1)\times \cdots\times(n-k)}{1\times\cdots\times(\ell-k)}%
%\]
\noindent which can only hold when $\ell=k$ for otherwise%
\[
\frac{\ell-k}{\ell!(n-\ell)!}<1\text{ and } \binom{n-k}{n-\ell} -1 \geq1.
\]
In particular, when $\ell>k$, the inequality (\ref{Kneser}) does not hold.
Neither does the inequality $\left\vert A_{0}\cdot Y\right\vert +\left\vert
{G}_{A_{0}\cdot Y}\right\vert \geq\left\vert G_{A_{0} \cdot Y} Y\right\vert
+\left\vert G_{A_{0} \cdot Y} A_{0}\right\vert $ since $G_{A_{0} \cdot Y} Y =
\{1,\ldots, \ell\}$ and $G_{A_{0} \cdot Y} A_{0} = A_{0}$. When $k=\ell$, the
set $A_{0}$ is a group isomorphic to the direct product $\mathfrak{S}%
_{k}\times\mathfrak{S}_{n-k}$ and we get $A_{0}\cdot Y=Y$ with ${G}%
_{A_{0}\cdot Y}=A_{0}$ in which case (\ref{Kneser}) becomes an equality.
\end{exa}

\begin{exa}
There exists another version of Kneser's theorem saying that for any two non
empty subsets $A,B$ of an Abelian group $G$ verifying $|A+B|<|A|+|B|$, the
stabilizer of $A+B$ is non trivial. This version is also no longer true in the
group action context. Indeed, let us consider $G$ the group of affine
transformations of the line on $\mathbb{F}_{7}$: $G=\{x\in\mathbb{F}%
_{7}\mapsto ax+b\in\mathbb{F}_{7},a\in\mathbb{F}_{7}^{\times},b\in
\mathbb{F}_{7}\}$, set $Y=\{1,2\}\subset\mathbb{F}_{7}$ and $A=\{x\mapsto
x,x\mapsto5x+3\}\subset G$. We have $AY=\{1,2,6\}$. But an easy computation
shows that the stabilizer of $AY$ in $G$ is the identity map $x\mapsto x$ even
though $|AY|<|A|+|Y|$.
\end{exa}

\subsection{Symmetry sets and upper bounds}

\label{subsection-murphy}

In this subsection, we use the symmetry sets introduced by Murphy in
\cite{Mur} to obtain results analogous to results about sets of small
doubling. Assume that the group $G$ acts on the set $X$ and consider a finite
nonempty subset $Y$ of $X$. Following Murphy ideas, for a real number
$\alpha\in[0,1]$, we introduce the \emph{symmetry set} of $Y$ in $G$ for
$\alpha$ is defined as%
\[
\mathrm{Sym}_{\alpha}(Y)=\{g\in G\mid\left\vert g\cdot Y\cap Y\right\vert
\geq\alpha\left\vert Y\right\vert \}\text{.}%
\]
We also introduce the \emph{weak stabilizer} of $Y$ as%
\[
\Gamma_{Y}=\{g\in G\mid g\cdot Y\cap Y\neq\emptyset\}=%
%TCIMACRO{\dbigcup \limits_{\alpha\in]0,1]}}%
%BeginExpansion
{\displaystyle\bigcup\limits_{\alpha\in]0,1]}}
%EndExpansion
\mathrm{Sym}_{\alpha}(Y).
\]
One immediately checks that $1\in\Gamma_{Y}$ and $g\in\mathrm{Sym}_{\alpha
}(Y)$ if and only if $g^{-1}\in\mathrm{Sym}_{\alpha}(Y)$.\ Also, if $G$ acts
on itself by left translation and $A\subset G$, we have $\Gamma_{A}=AA^{-1}%
$.\ Observe also that $\mathrm{Sym}_{1}(Y)=G_{Y}$ is the stabilizer of $Y$ in
$G$. In general we always have $G_{Y}\subset\mathrm{Sym}_{\alpha}(Y)$ for any
$\alpha\in[0,1]$ and more generally $\mathrm{Sym}_{\alpha}(Y)\subset
\mathrm{Sym}_{\alpha^{\prime}}(Y)$ for $\alpha^{\prime}\leq\alpha$. Therefore
the set of subsets $(\mathrm{Sym}_{\alpha}(Y))_{\alpha\in\lbrack0,1]}$
decreases from $G$ to $G_{Y}$ when $\alpha$ increases in $[0,1] $. The set%
\[
\left\{  \frac{\left\vert g\cdot Y\cap Y\right\vert }{\left\vert Y\right\vert
}\mid g\in\Gamma_{Y}\right\}  \subset\mathbb{Q}_{>0}%
\]
is discrete and not empty. Thus it admits a minimum $\alpha_{0}$ and we have
$\Gamma_{Y}=\mathrm{Sym}_{\alpha_{0}}(Y)$.

When $(\rho,V)$ is a linear representation of $G$ such that $V\neq\{0\}$, we
define similarly for any $k$-subspace $W\neq\{0\}$ of $V$%
\begin{align*}
\mathrm{Sym}_{\alpha}(W)  &  =\{g\in G\mid\dim g\cdot W\cap W \geq\alpha\dim
W\}\text{ and }\\
\Gamma_{W}  &  =\{g\in G\mid g\cdot W\cap W\neq\{0\}\}.
\end{align*}

We also have $1\in\mathrm{Sym}_{\alpha}(W)$ and $g\in\mathrm{Sym}_{\alpha}(W)
$ if and only if $g^{-1}\in\mathrm{Sym}_{\alpha}(W)$.\ 

\bigskip

In this section, we examine what kind of information can be extracted when
some assumptions are imposed on the cardinality ratio $\frac{\left\vert A\cdot
Y\right\vert }{\left\vert Y\right\vert }$ (or the dimension ratio $\frac
{\dim(A\cdot Y)}{\dim(Y)}$). This problem was addressed in detail by Murphy in
\cite{Mur} for group action setting.\ Let us start by recalling Theorem 1 of
\cite{Mur} and state its linear version.

\begin{prop}
\begin{enumerate}
\item Assume that $\left\vert A\cdot Y\right\vert =\left\vert Y\right\vert $.
Then $H=\langle A^{-1}A\rangle$\footnote{Here $\langle A^{-1}A\rangle$ means
the subgroup of $G$ generated by $A^{-1}A $.} is a subgroup of $G_{Y}$ and $Y$
decomposes into $H$-orbits.

\item Assume that $\dim\langle A\cdot W\rangle=\dim W$. Then $H=\langle
A^{-1}A\rangle$ is a subgroup of $G_{W}$.\ When $k$ has characteristic zero
and $H$ is finite, the $k$-space $W$ decomposes into irreducible
representations for the group $H$.
\end{enumerate}
\end{prop}

\begin{proof}
1: For any $a\in A$, we have $1\in a^{-1}A$ and $a^{-1}A\cdot Y=Y$ because
$Y\subset a^{-1}A\cdot Y$ and $\left\vert a^{-1}A\cdot Y\right\vert
=\left\vert A\cdot Y\right\vert =\left\vert Y\right\vert $.\ This shows that
$A^{-1}A\cdot Y=Y$ and thus the desired inclusion $\langle A^{-1}%
A\rangle\subset G_{Y}$. Since $H$ is a subgroup of $G_{Y}$, it acts on $Y$
which yields the decomposition in $H$-orbits.
2: We get similarly $A^{-1}A\cdot W=W$ and the decomposition of $W$ in
irreducible representations for the finite group $H$ follows from the
semisimplicity of its representation theory in characteristic zero.
\end{proof}

\subsubsection{Small growing sets}

In his article \cite{Mur}, Murphy extends Ruzsa's triangle inequality, Ruzsa's
covering lemma and Balog-Szemer\'{e}di-Gowers theorem to the context of group
actions. Here, we extend results on small growing subsets: we examine cases
where the hypotheses of the previous proposition are relaxed. In the following
$\alpha$ is a fixed real number in $]0,1]$.

\begin{lem}
\label{Lemma_Sym}

\begin{enumerate}
\item Assume that $A\subset G$ and $Y\subset X$ are finite and nonempty and
satisfy $\left\vert A\cdot Y\right\vert \leq(2-\alpha)\left\vert Y\right\vert
$. Then $A^{-1}A\subset\mathrm{Sym}_{\alpha}(Y)$.

\item Assume $A\subset G$ and $W$ is a finite-dimensional $k$-subspace of $V$
such that $\dim\langle A\cdot W\rangle\leq(2-\alpha)\dim W$. Then
$A^{-1}A\subset\mathrm{Sym}_{\alpha}(W)$.
\end{enumerate}
\end{lem}

\begin{proof}
1: Consider $a,b$ in $A$.\ Since we have $\left\vert a\cdot Y\right\vert
=\left\vert b\cdot Y\right\vert =\left\vert Y\right\vert ,$ $a\cdot Y\subset
A\cdot Y$, $b\cdot Y\subset A\cdot Y$ and $\left\vert A\cdot Y\right\vert
\leq(2-\alpha)\left\vert Y\right\vert ,$ we must have $\left\vert (a\cdot
Y)\cap(b\cdot Y)\right\vert \geq\alpha\left\vert Y\right\vert $.\ We thus
obtain $\left\vert (b^{-1}a\cdot Y)\cap Y\right\vert \geq\alpha\left\vert
Y\right\vert $ and the desired inclusion $A^{-1}A\subset\mathrm{Sym}_{\alpha
}(Y)$.
2: This works similarly using Grassmann formula.
\end{proof}

\bigskip

Given a subset $S$ of $G$, we denote by $\langle S\rangle$ the subgroup of $G
$ generated by the elements in $S$. The next proposition extends a standard
result of Fre\u{\i}man~\cite{freiman}.

\begin{prop}
\ \label{Prop_3/2}

\begin{enumerate}
\item Assume that $A\subset G$ and $Y\subset X$ are nonempty, $Y$ is finite
and $A$ and $Y$ satisfy $\left\vert A^{-1}\cdot Y\right\vert \leq
\frac{3-\alpha}{2}\left\vert Y\right\vert $. Then $(AA^{-1})^{2}$ and
$AA^{-1}$ are contained in $\mathrm{Sym}_{\alpha}(Y)$.

\item Assume that $A\subset G$ is non empty and $W$ is a nonzero finite
dimensionnal $k$-subspace of $V$ such that $\dim\langle A^{-1}\cdot
W\rangle\leq\frac{3-\alpha}{2}\dim W$. Then $(AA^{-1})^{2}$ is contained in
$\mathrm{Sym}_{\alpha}(W).$
\end{enumerate}
\end{prop}

\begin{proof}
1: Consider $u=ab^{-1}$ in $AA^{-1}$ with $a,b$ in $A$. We have
\begin{multline*}
\left\vert (a^{-1}\cdot Y)\cap(b^{-1}\cdot Y)\right\vert =\left\vert
a^{-1}\cdot Y\right\vert +\left\vert b^{-1}\cdot Y\right\vert -\left\vert
(a^{-1}\cdot Y)\cup(b^{-1}\cdot Y)\right\vert \geq\\
2\left\vert Y\right\vert -\left\vert A^{-1}\cdot Y\right\vert \geq
\frac{1+\alpha}{2}\left\vert Y\right\vert
\end{multline*}
where the second inequality follows from the inclusions $a^{-1}\cdot Y\subset
A^{-1}\cdot Y$ and $b^{-1}\cdot Y\subset A^{-1}\cdot Y$ together with the
hypothesis $\left\vert A^{-1}\cdot Y\right\vert \leq\frac{3-\alpha}%
{2}\left\vert Y\right\vert $. We thus get%
\[
\left\vert Y\cap u\cdot Y\right\vert \geq\frac{1+\alpha}{2}\left\vert
Y\right\vert .
\]
For any $v\in AA^{-1}$, we get similarly%
\[
\left\vert v^{-1}\cdot Y\cap Y\right\vert =\left\vert Y\cap v\cdot
Y\right\vert \geq\frac{1+\alpha}{2}\left\vert Y\right\vert .
\]
This implies that both sets $Y\cap u\cdot Y$ and $v^{-1}\cdot Y\cap Y$
intersect non trivially in $Y$ and
\[
\left\vert u\cdot Y\cap v^{-1}\cdot Y\cap Y\right\vert \geq\left\vert Y\cap
u\cdot Y\right\vert +\left\vert v^{-1}\cdot Y\cap Y\right\vert -\left\vert
Y\right\vert \geq\alpha\left\vert Y\right\vert .
\]
Therefore we obtain that $\left\vert vu\cdot Y\cap Y\right\vert \geq
\alpha\left\vert Y\right\vert $ and the product $vu$ of any two elements $u,v$
in $AA^{-1}$ belongs to $\mathrm{Sym}_{\alpha}(Y)$. In particular, by taking
$v=1\in AA^{-1}$, we get that $AA^{-1}$ is contained in $\mathrm{Sym}_{\alpha
}(Y)$.
2: The proof can be easily adapted to the context of a the linear
representation $V$ of $G$.
\end{proof}

\bigskip

\begin{rem}
\ 

\begin{enumerate}
\item When $G$ acts on itself by left translation and $Y=A$, we have
$\Gamma_{A}=AA^{-1}$ and the hypothesis $\left\vert A^{-1}\cdot Y\right\vert
<\frac{3}{2}\left\vert Y\right\vert $ implies that $\langle AA^{-1}\rangle
_{G}\subset AA^{-1}$, that is $AA^{-1}$ is itself a subgroup of $G$. Indeed,
for some $\alpha$, $(AA^{-1})^{2} \subset\mathrm{Sym}_{\alpha}(Y)
\subset\Gamma_{A} = AA^{-1}$.

\item If we assume $\left\vert A\cdot Y\right\vert <\frac{3}{2}\left\vert
Y\right\vert $, we get similarly that $(A^{-1}A)^{2}$ is contained $\Gamma
_{Y}$.
\end{enumerate}
\end{rem}

Assertion 1 of the previous remark suggests the following corollary of
Proposition~\ref{Prop_3/2}.

\begin{cor}
Assume that $A\subset G$ and $Y\subset X$ are nonempty with $Y$ a finite set
and that there exists $\alpha\in]0,1[$ such that%
\[
\mathrm{Sym}_{\alpha}(Y)\subset AA^{-1}\text{ and }\left\vert A^{-1}\cdot
Y\right\vert \leq\frac{3-\alpha}{2}\left\vert Y\right\vert .
\]
Then $AA^{-1}$ is a subgroup of $G$.
\end{cor}

\begin{proof}
By Proposition \ref{Prop_3/2}, we get $(AA^{-1})^{2}\subset\mathrm{Sym}%
_{\alpha}(Y)\subset AA^{-1}$. Therefore, $AA^{-1}$ is a subgroup of $G$.
\end{proof}

\subsection{Action of a product subset of $G$ on a subset of $X$}

\label{subsection-ruzsa}

Assume that $G$ acts on the set $X$. We now address the question of
determining an upper bound of $AB\cdot Y$ when $A,B$ are nonempty finite
subsets of $G$ and $Y$ a finite subset of $X$. This is a group action version
of Theorem~9.2 of \cite{Ruz}.

\begin{thm}
\label{Th_triple}With the previous notation we have%
\begin{equation}
\left\vert AB\cdot Y\right\vert ^{2}\leq\left\vert AB\right\vert \left\vert
B\cdot Y\right\vert \max_{b\in B}\{\left\vert Ab\cdot Y\right\vert
\}.\label{triple}%
\end{equation}
In particular, when the elements of $A$ commute with those of $B$ we have%
\[
\left\vert AB\cdot Y\right\vert ^{2}\leq\left\vert AB\right\vert \left\vert
B\cdot Y\right\vert \left\vert A\cdot Y\right\vert .
\]

\end{thm}

\begin{proof}
We proceed by induction on $\left\vert B\right\vert $. When $B=\{b\}$, we
obtain
\[
\left\vert Ab\cdot Y\right\vert ^{2}\leq\left\vert Ab\right\vert \left\vert
b\cdot Y\right\vert \max_{b\in B}\{\left\vert Ab\cdot Y\right\vert \}
\]
by observing that $\left\vert Ab\cdot Y\right\vert \leq\left\vert
Ab\right\vert \left\vert Y\right\vert $ and $\left\vert b\cdot Y\right\vert
=\left\vert Y\right\vert $.\ Now assume $\left\vert B\right\vert >1$, set
$m=\max_{u\in B}\{\left\vert Au\cdot Y\right\vert \}$ and fix $b\in B$ such
that $m=\left\vert Ab\cdot Y\right\vert $.\ Write $B=B^{\prime}\cup
\{b\}$.\ Set $A=\{a_{1},\ldots,a_{r}\}$ and $Y=\{y_{1},\ldots,y_{s}\}$. We
have $AB=AB^{\prime}\cup Ab$. There exists a subset $A^{\flat}$ of $A$ such
that
\[
AB=AB^{\prime}%
%TCIMACRO{\tbigsqcup \limits_{a\in A^{\flat}}}%
%BeginExpansion
{\textstyle\bigsqcup\limits_{a\in A^{\flat}}}
%EndExpansion
ab.
\]
Similarly, there exists a subset $Y^{\flat}$ of $Y$ such that
\[
B\cdot Y=(B^{\prime}\cdot Y)%
%TCIMACRO{\tbigsqcup \limits_{y\in Y^{\flat}}}%
%BeginExpansion
{\textstyle\bigsqcup\limits_{y\in Y^{\flat}}}
%EndExpansion
b\cdot y.
\]
We get%
\begin{multline*}
AB\cdot Y=(AB^{\prime}\cdot Y)%
%TCIMACRO{\tbigcup \limits_{a\in A^{\flat}}}%
%BeginExpansion
{\textstyle\bigcup\limits_{a\in A^{\flat}}}
%EndExpansion
(ab\cdot Y)=(AB^{\prime}\cdot Y)%
%TCIMACRO{\tbigcup \limits_{a\in A^{\flat}}}%
%BeginExpansion
{\textstyle\bigcup\limits_{a\in A^{\flat}}}
%EndExpansion
(aB\cdot Y)=
\\ (AB^{\prime}\cdot Y)\!%
%TCIMACRO{\tbigcup \limits_{a\in A^{\flat}}}%
%BeginExpansion
{\textstyle\bigcup\limits_{a\in A^{\flat}}}
%EndExpansion
\! (aB^{\prime}\cdot Y)%
%TCIMACRO{\tbigcup \limits_{a\in A^{\flat}}}%
%BeginExpansion
{\textstyle\bigcup\limits_{a\in A^{\flat}}}
%EndExpansion%
%TCIMACRO{\tbigcup \limits_{y\in Y^{\flat}}}%
%BeginExpansion
{\textstyle\bigcup\limits_{y\in Y^{\flat}}}
%EndExpansion
(ab\cdot y).
\end{multline*}
Since we have $%
%TCIMACRO{\tbigcup \limits_{a\in A^{\flat}}}%
%BeginExpansion
{\textstyle\bigcup\limits_{a\in A^{\flat}}}
%EndExpansion
(aB^{\prime}\cdot Y)\subset AB^{\prime}\cdot Y$, we can write%
\[
AB\cdot Y=(AB^{\prime}\cdot Y)%
%TCIMACRO{\tbigcup \limits_{a\in A^{\flat}}}%
%BeginExpansion
{\textstyle\bigcup\limits_{a\in A^{\flat}}}
%EndExpansion%
%TCIMACRO{\tbigcup \limits_{y\in Y^{\flat}}}%
%BeginExpansion
{\textstyle\bigcup\limits_{y\in Y^{\flat}}}
%EndExpansion
(ab\cdot y).
\]
By the previous decomposition, there exists $X\subset A^{\flat}\times
Y^{\flat}$ such that%
\[
AB\cdot Y=(AB^{\prime}\cdot Y)%
%TCIMACRO{\tbigsqcup \limits_{(a,y)\in X}}%
%BeginExpansion
{\textstyle\bigsqcup\limits_{(a,y)\in X}}
%EndExpansion
(ab\cdot y).
\]
Set $\alpha=\left\vert X\right\vert $, $\beta=\left\vert A^{\flat}\right\vert
$ and $\gamma=\left\vert Y^{\flat}\right\vert $.\ Since $\left\vert AB\cdot
Y\right\vert =\left\vert AB^{\prime}\cdot Y\right\vert +\alpha$, the desired
inequality (\ref{triple}) is equivalent to
\begin{equation}
(\left\vert AB^{\prime}\cdot Y\right\vert +\alpha)^{2}\leq(\left\vert
AB^{\prime}\right\vert +\beta)(\left\vert B^{\prime}\cdot Y\right\vert
+\gamma)m. \label{tripleref}%
\end{equation}
By the induction hypothesis, we have
\begin{equation}
\left\vert AB^{\prime}\cdot Y\right\vert ^{2}\leq\left\vert AB^{\prime
}\right\vert \left\vert B^{\prime}\cdot Y\right\vert m. \label{triplerec}%
\end{equation}
because $\max_{u\in B^{\prime}}\{\left\vert Au\cdot Y\right\vert )\}\leq
\max_{u\in B}\{\left\vert Au\cdot Y\right\vert )\}=m$.
Moreover we have $%
%TCIMACRO{\tbigsqcup \limits_{(a,y)\in X}}%
%BeginExpansion
\!\!\!{\textstyle\bigsqcup\limits_{(a,y)\in X}}
%EndExpansion
(ab\cdot y)\subset Ab\cdot Y$ and therefore $\alpha\leq m$. Since $X\subset
A^{\flat}\times Y^{\flat}$, we have also $\alpha\leq\beta\gamma$. Hence
$\alpha^{2}\leq m\beta\gamma.$ By multiplying with (\ref{triplerec}), this
gives
\[
\alpha^{2}\left\vert AB^{\prime}\cdot Y\right\vert ^{2}\leq\left\vert
AB^{\prime}\right\vert \left\vert B^{\prime}\cdot Y\right\vert m^{2}%
\beta\gamma.
\]
Therefore%
\[
\alpha\left\vert AB^{\prime}\cdot Y\right\vert \leq m\sqrt{\gamma\left\vert
AB^{\prime}\right\vert \times\beta\left\vert B^{\prime}\cdot Y\right\vert
}\leq m\frac{\gamma\left\vert AB^{\prime}\right\vert +\beta\left\vert
B^{\prime}\cdot Y\right\vert }{2}.
\]
So
\[
2\alpha\left\vert AB^{\prime}\cdot Y\right\vert \leq m\gamma\left\vert
AB^{\prime}\right\vert +m\beta\left\vert B^{\prime}\cdot Y\right\vert .
\]
Combining this last inequality with $\alpha^{2}\leq m\beta\gamma$ and
(\ref{triplerec}), we finally get%
\begin{multline*}
(\left\vert AB^{\prime}\cdot Y\right\vert +\alpha)^{2}=\left\vert AB^{\prime
}\cdot Y\right\vert ^{2}+2\alpha\left\vert AB^{\prime}\cdot Y\right\vert
+\alpha^{2}\leq\\
m\left\vert AB^{\prime}\right\vert \left\vert B^{\prime}\cdot Y\right\vert
+m\gamma\left\vert AB^{\prime}\right\vert +m\beta\left\vert B^{\prime}\cdot
Y\right\vert +m\beta\gamma=(\left\vert AB^{\prime}\right\vert +\beta
)(\left\vert B^{\prime}\cdot Y\right\vert +\gamma)m
\end{multline*}
as desired.
\end{proof}

\section{Submodular functions}

\label{sec-submodular}

The goal of this section is to show how techniques based on submodular
functions are efficient methods to obtain results in the group action setting.

\subsection{Background}

Consider a set $S$ (in the sequel $S$ could be a group $G$ or the set $X$ on
which $G$ acts).\ Let $\mathcal{P}(S)$ be the power set of $S$.

\begin{definition}
The map $f:\mathcal{P}(S)\rightarrow\mathbb{R}$ is said to be submodular when
\begin{equation}
f(A\cap B)+f(A\cup B)\leq f(A)+f(B)\label{Ineg}%
\end{equation}
for any subsets $A$ and $B$ in $\mathcal{P}(S)$.

The submodular function $f$ is said increasing when $f(A)\leq f(B)$ for any
subsets $A\subset B\subset S$.

The submodular function $f$ is said $G$-invariant when $f(gA)=f(A)$ for any
subsets $A\subset S$ et any $g \in G$.
\end{definition}

Very often, we shall consider submodular functions defined on the set
$\mathcal{P}_{\mathrm{fin}}(S)$ of finite subsets in $S$ rather than on
$\mathcal{P}(S)$. When $S$ is finite, one can check that $f$ is submodular if
and only if for any subsets $A_{1}\subset A_{2}$ of $\mathcal{P}(S)$ and any
$s\in S\setminus A_{2}$, we have%
\begin{equation}
f(A_{1}\cup\{s\})-f(A_{1})\geq f(A_{2}\cup\{s\})-f(A_{2}).\label{AlterSubmod}%
\end{equation}

Let us now introduce examples of submodular functions relevant for our purposes.

%\subsubsection{Weight functions}
%Assume we have a nonnegative weight function $W:S\rightarrow\mathbb{R}_{\geq
%0}$ and set $W(s)=w_{s}$ for any $s\in W$. Then the map $f:\mathcal{P}%
%_{\mathrm{fin}}(S)\rightarrow\mathbb{R}_{\geq0}$ defined by
%\[
%f(A)=\sum_{s\in A}w_{s}%
%\]
%is submodular. In fact the inequality (\ref{Ineg}) becomes here an
%equality.\ When $W(s)=1$ for any $s\in S$, we get $f(S)=\left\vert
%S\right\vert $.$\ $Also when $\sum_{s\in S}w_{s}=1$, the map $f$ can be
%interpreted as a discrete probability measure on the set $S$.

\subsection{Combinations of submodular functions}

\label{Subsec_Comb_SMF}We know record the two following easy propositions.

\begin{prop}
The set of nonnegative submodular functions defined from a set $S$ is a cone:
given $f$ and $g$ nonnegative submodular on $\mathcal{P}(S)$ and $(\lambda
,\mu)\in\mathbb{R}_{\geq0}$, the map%
\[
\lambda f+\mu g
\]
is yet submodular nonnegative.
\end{prop}

Now assume $f$ is submodular ($f$ is not assumed nonnegative here) and $u$ is
a modular map defined on $\mathcal{P}(S)$, that is satisfying%
\[
u(A\cup B)+u(A\cap B)=u(A)+u(B).
\]

\begin{prop}
For any real $\lambda\in\mathbb{R}$, the map $f-\lambda u$ is submodular on
$\mathcal{P}(S)$.
\end{prop}

%\begin{proof}
%Consider $A$ and $B$ subsets of $S$.\ We have
%\begin{equation}
%f(A\cup B)+f(A\cap B)\leq f(A)+f(B) \label{ineq2}%
%\end{equation}
%because $f$ is submodular and
%\[
%\lambda u(A\cup B)+\lambda u(A\cap B)=\lambda u(A)+\lambda u(B)
%\]
%by definition of the map $u$. By subtracting the last inequality to
%(\ref{ineq2}), we get the submodularity of $f-\lambda u$.
%\end{proof}

\subsection{Fundamental examples of submodular functions}

\label{ssec-ex-fonda}

In this subsection, we give four examples of submodular functions. These
functions will be studied in detail in the next sections.

\subsubsection{Group action and graph cut type submodular function}

Let $G$ be a finite group acting on the finite set $X$. For any subset
$Y\subset X$, set%
\[
E_{Y}=\{(g,y)\in G\times Y\mid g\cdot y\notin Y\}.
\]
Consider the cut function
\begin{equation}
f:\left\{
\begin{array}
[c]{l}%
\mathcal{P}(X)\rightarrow\mathbb{Z}_{\geq0}\\
Y\longmapsto\left\vert E_{Y}\right\vert
\end{array}
\right. \label{cut_function}%
\end{equation}

\begin{prop}
\label{Prop_Submod_cut}The previous function $f$ is $G$-invariant submodular
and nonnegative.
\end{prop}

\begin{proof}
Consider two subsets $Y_{1}$ and $Y_{2}$ of $X$ such that $Y_{1}\subset Y_{2}$ and
$y_{0}\in X\setminus Y_{2}$. Then $E_{Y_{1}\cup\{y_{0}\}}$ is the disjoint union of%
\[
\{(g,y)\in G\times Y_{1}\mid g\cdot y\notin
Y_{1}\}\setminus\{(g,y)\in G\times Y_{1}\mid g\cdot y=y_{0}\}\, \textrm{and }
%%TCIMACRO{\tbigsqcup }%
%%BeginExpansion
%{\textstyle\bigsqcup}
%%EndExpansion
\{(g,y_{0})\mid g\cdot y_{0}\notin Y_{1}\cup\{y_{0}\}\}.
\]
This gives%
\[
f(E_{Y_{1}\cup\{y_{0}\}})=f(E_{Y_{1}})+\left\vert \{(g,y_{0})\mid g\cdot
y_{0}\notin Y_{1} \cup\{y_{0}\}\}\right\vert -\left\vert G_{y_{0}} \right\vert
\left\vert \mathcal{O}_{y_{0}}\cap Y_{1} \right\vert .
\]
Similarly, we have
\[
f(E_{Y_{2}\cup\{y_{0}\}})=f(E_{Y_{2}})+\left\vert \{(g,y_{0})\mid g\cdot
y_{0}\notin Y_{2}\cup\{y_{0}\}\}\right\vert -\left\vert G_{y_{0}} \right\vert
\left\vert \mathcal{O}_{y_{0}}\cap Y_{2}\right\vert .
\]
Now, the assumption $Y_{1}\subset Y_{2}$ implies the set inclusions
\begin{multline*}
\{(g,y_{0})\mid g\cdot y_{0}\notin Y_{2} \cup\{y_{0}\}\}\subset\{(g,y_{0})\mid
g\cdot y_{0} \notin Y_{1} \cup\{y_{0}\} \}\text{ and } \mathcal{O}_{y_{0}}\cap
Y_{1}\subset\mathcal{O}_{y_{0}}\cap Y_{2}.
\end{multline*}
This gives%
\[
f(E_{Y_{1}\cup\{y_{0}\}})-f(E_{Y_{1}})\geq f(E_{Y_{2}\cup\{y_{0}%
\}})-f(E_{Y_{2}})
\]
and $f$ is submodular by (\ref{AlterSubmod}). Moreover, the function $f$ is
clearly nonnegative. Finally, for any $Y\subset X$ and any $g_{0}\in G$ the
map%
\[
\chi_{g_{0}}:\left\{
\begin{array}
[c]{l}%
E_{Y}\rightarrow E_{g_{0}\cdot Y}\\
(g,y)\longmapsto(g_{0}gg_{0}^{-1},g_{0}\cdot y)
\end{array}
\right.
\]
is a bijection which implies the desired equality $f(g_{0}\cdot Y)=f(Y)$.
\end{proof}

\bigskip

\begin{rem}
When the action of $G$ on $X$ is free, it can be represented by an oriented
graph $\Gamma=(X,E)$ with set of vertices $X$ and set of arrows $x\rightarrow
x^{\prime}$ when there exists $g\in G$ such that $x^{\prime}=g\cdot x$.
Observe that such an element $g$ is then unique by assumption. Then the
previous function $f$ becomes the cut function of $\Gamma$ which is classical
in graph theory and known to be submodular.
\end{rem}

\subsubsection{Action on a fixed set or subspace}

\label{sssec-cy-gammay}

Assume that $G$ acts on $X$.\ Fix $Y$ a finite subset of $X$ and $\lambda$ a
real. Let $\mathcal{P}_{\mathrm{fin}}(G)$ be the set of finite subsets in $G$.
Then the map%
\[
c_{Y}:\left\{
\begin{array}
[c]{l}%
\mathcal{P}_{\mathrm{fin}}(G)\rightarrow\mathbb{R}\\
A\longmapsto\left\vert A\cdot Y\right\vert -\lambda\left\vert A\right\vert
\end{array}
\right.
\]
is $G$-invariant submodular for every $\lambda$ and increasing when
$\lambda=0$.\ Indeed, we have for any two finite subsets $A$ and $B$ of $G$
and any $g\in G$
\[
(A\cap B)\cdot Y\subset(A\cdot Y)\cap(B\cdot Y)\text{ and }(A\cup B)\cdot
Y=(A\cdot Y)\cup(B\cdot Y),
\]
and $|gA\cdot Y|=|A\cdot Y|$ and $|gA|=|A|$. Similarly, when $(\rho,V)$ is a
linear representation of $G$ and $W$ a fixed subspace of $V$, the map%
\[
\gamma_{W}:\left\{
\begin{array}
[c]{l}%
\mathcal{P}_{\mathrm{fin}}(G)\rightarrow\mathbb{R}\\
A\longmapsto\dim(A\cdot W)-\lambda|A|
\end{array}
\right.
\]
is $G$-invariant submodular for every $\lambda$ and increasing when
$\lambda=0$ because%
\[
\langle(A\cap B)\cdot W\rangle\subset\langle A\cdot W\rangle\cap\langle B\cdot
W\rangle\quad\text{ and }\quad\langle(A\cup B)\cdot W\rangle=\langle A\cdot
W\rangle+\langle B\cdot W\rangle.
\]

\subsubsection{Action of a fixed subset in a group}

\label{sssec-dA}

When $G$ acts on $X$ and $A$ is a fixed finite subset of $G$ and
$\lambda\mathbb{\ }$a fixed real, we can alternatively consider the map%
\[
d_{A}:\left\{
\begin{array}
[c]{l}%
\mathcal{P}_{\mathrm{fin}}(X)\rightarrow\mathbb{R}\\
Y\longmapsto\left\vert A\cdot Y\right\vert -\lambda\left\vert Y\right\vert
\end{array}
\right.
\]
defined on the set $\mathcal{P}_{\mathrm{fin}}(X)$ of finite subsets of $X$.
This gives yet a submodular function since for any $Y,Z$ in $\mathcal{P}%
_{\mathrm{fin}}(X)$, we have%
\[
A\cdot(Y\cap Z)\subset(A\cdot Y)\cap(A\cdot Z)\text{ and }A\cdot(Y\cup
Z)=(A\cdot Y)\cup(A\cdot Z).
\]

\begin{rem}
When $G$ is Abelian, the submodular function $d_{A}$ is $G$-invariant since
for every finite subset $Y$ of $X$ and every $g\in G$, we have $|A\cdot g\cdot
Y|=|g\cdot A\cdot Y|=|A\cdot Y|$ and $|g\cdot Y|=|Y|$. But this is not
necessarily the case when $G$ is not Abelian as illustrated by the example below.
\end{rem}

\begin{exa}
Assume that $G=\mathfrak{S}_{5}$ regarded as the symmetric group permuting the
set $\{1,2,3,4,5\}$. Set $Y=\{1,2\}$ and take for the subset $A$ the subgroup
of $G$ of permutations of the set $\{3,4,5\}$. Now for $g\in G$ such that
$g(1)=5,g(2)=4,g(3)=3,g(4)=2$ and $g(5)=1$, we have
\[
\left\vert A\cdot Y\right\vert =\left\vert Y\right\vert =2
\]
but%
\[
\left\vert A\cdot g\cdot Y\right\vert =\left\vert A\cdot\{4,5\}\right\vert
=\left\vert \{3,4,5\}\right\vert =3.
\]

\end{exa}

\section{Fragments and atoms}

\label{sec-fragment-atom}

In this section, we derive some minimisation properties of submodular
functions and their applications to the case of $G$-invariant submodular maps.

\subsection{Definitions and general properties}

In this paragraph, we fix a submodular function $f$ defined on $\mathcal{P}%
(S)$ such that $m=\min_{Y\neq\emptyset\in\mathcal{P}_{\mathrm{fin}}(S)}f(Y)$
exists. Then a \emph{fragment} for $f$ is a nonempty finite subset $Y$ of $S$
such that $f(Y)=m$. An \emph{atom} for $f$ is a fragment of minimum
cardinality. Observe that there exists at least one fragment and one atom by
the hypotheses on $f$. Moreover, by definition, \emph{all the atoms have the
same finite cardinality}.

\begin{lem}
\label{Lem_EmptyInter}Assume $A_{1}$ and $A_{2}$ are two atoms for the
submodular function $f$. Then $A_{1}=A_{2}$ or $A_{1}\cap A_{2}=\emptyset$.
\end{lem}

\begin{proof}
Assume $A_{1}\cap A_{2}$ is not empty.\ Since $f$ is a submodular function on
$\mathcal{P}_{\mathrm{fin}}(S)$, we can write%
\[
f(A_{1}\cap A_{2})+f(A_{1}\cup A_{2})\leq f(A_{1})+f(A_{2})=2m
\]
by using that $A_{1}$ and $A_{2}$ are atoms. We have $f(A_{1}\cap A_{2})\geq
m$ and $f(A_{1}\cup A_{2}) \geq m$ since $m=\min_{Y\neq\emptyset\in
\mathcal{P}_{\mathrm{fin}}(S)}f(Y)$, we get $f(A_{1}\cap A_{2})=f(A_{1}\cup
A_{2})=m$. Hence both $A_{1}\cup A_{2}$ and $A_{1}\cap A_{2}$ are fragments
for $f$.\ Now, observe that $A_{1}\cap A_{2}\subset A_{1}$. Thus by minimality
of the cardinality of an atom, we have $\left\vert A_{1}\cap A_{2}\right\vert
=\left\vert A_{1}\right\vert $ and therefore $A_{1}\cap A_{2}=A_{1}$ which
means that $A_{1} \subset A_{2}$. But $A_{1}$ and $A_{2}$ have same
cardinality since they are atoms. So $A_{1}=A_{2}$.
\end{proof}

\subsection{Invariant submodular functions on groups}

Let $G$ be a group and $f:\mathcal{P}_{\mathrm{fin}}(G)\rightarrow\mathbb{R}$
a submodular function.\ Recall that it is said $G$-invariant when $f(gA)=f(A)$ for
any $g\in G$ and any finite subset $A\subset G$.

\begin{prop}
\label{Prop_H}Assume that $f$ is a $G$-invariant submodular map such that
$m=\min_{A\neq\emptyset\in\mathcal{P}_{\mathrm{fin}}(G)}f(A)$ exists. Then,
there exists a unique atom $H$ for $f$ containing $1$. Moreover $H$ is a
finite subgroup of $G$, the atoms of $G$ are the left cosets $gH$ with $g\in
G$ and they yield a partition of $G$.
\end{prop}

\begin{proof}
The existence of an atom is obtained as in the previous paragraph. Now, if $A$
is an atom, since it is nonempty, we get that $a^{-1}A$ is also an atom for
any $a\in A$ because $f(a^{-1}A)=m$. Then $H=a^{-1}A$ is an atom containing
$1$. Let $H^{\prime}$ be another atom containing $1$. Then $H\cap H^{\prime}$
is nonempty, thus by Lemma \ref{Lem_EmptyInter}, we must have $H=H^{\prime}$
which proves that there exists indeed a unique atom $H$ containing $1$.\ Given
$h\in H$, we show similarly that $h^{-1}H$ is an atom containing $1$ so that
$h^{-1}H=H.$ Therefore, for any $h,h^{\prime}\in H$ we get that $h^{-1}%
h^{\prime}$ belongs to $H$ which shows that $H$ is a subgroup of $G$ (finite
by definition of an atom for $f$). Let $A$ be an atom for $H$. Then, for any $a\in A$, the
atom $a^{-1}A$ coincides with $H$ because it contains $1$. Thus, $A=aH$ is a
left coset of $H$.\ It is then well-known that the left cosets of $H$ give a
partition of $G$.
\end{proof}

\subsection{Invariant submodular functions for group actions}

\label{Subsec_leftinvariant}Assume that $G$ acts on the set $X$ and consider
$f:\mathcal{P}_{\mathrm{fin}}(X)\rightarrow\mathbb{R}$ a submodular function.
Recall that the function $f$ is said $G$-invariant if for any $g\in G$ and any
$Y\subset X$, we have $f(g\cdot Y)=f(Y)$. Assume that $m=\inf_{Y\neq
\emptyset\in\mathcal{P}_{\mathrm{fin}}(X)}f(Y)$ exists and $f$ is $G$-invariant. 
In this case, we get by Lemma \ref{Lem_EmptyInter} that for any
atom $Y_{0}$ and any $g\in G$%
\[
g\cdot Y_{0}\text{ is an atom such that }g\cdot Y_{0}=Y_{0}\text{ or }g\cdot
Y_{0}\cap Y_{0}=\emptyset\text{.}%
\]
Let $\mathcal{A}$ be the set of atoms for $f$. We thus get an action of the
group $G$ on the set of atoms $\mathcal{A}$. Now given any element $y_{0}$ in
the atom $Y_{0}$, we obtain the inclusion $G_{y_{0}}\subset G_{Y_{0}}$ of the
stabilizers of $y_{0}$ and $Y_{0}$ for the action of $G$ on $X$. Indeed, for
any $g\in G_{y_{0}}$, we have $y_{0}=g\cdot y_{0}\in g\cdot Y_{0}$ and also
$y_{0}\in Y_{0}$. Therefore $g\cdot Y_{0}\cap Y_{0}\neq\emptyset$ and $g\cdot
Y_{0}=Y_{0}$ which means that $g$ belongs to $G_{Y_{0}}$. Observe also that if
$y_{0}$ belongs to the atom $Y_{0}$, then any element $g\cdot y_{0}$ also
belongs to an atom (because $g\cdot y_{0}$ belongs to $g\cdot Y_{0}$)$.$ We
will call the set
\[
\mathcal{C}(X)=%
%TCIMACRO{\tcoprod \limits_{Y_{0}\in\mathcal{A}}}%
%BeginExpansion
{\textstyle\coprod\limits_{Y_{0}\in\mathcal{A}}}
%EndExpansion
Y_{0}%
\]
the \emph{core} of $X$.\ The action of $G$ on $X$ restricts to an action on
$\mathcal{C}(X)$ and thus, the set $\mathcal{C}(X)$ is a disjoint union of
orbits for the action of $G$ on $X$. Moreover, for any such orbit
$\mathcal{O}$ and any atom $Y_{0}$, we have
\[
\mathcal{O}\cap Y_{0}=\emptyset\text{ or }\mathcal{O}\cap Y_{0}=\{g\cdot
y\mid\overline{g}\in G_{Y_{0}}/G_{y_{0}}\}\text{ with }y_{0}\in\mathcal{O}\cap
Y_{0}\,,
\]
that is, $\mathcal{O}\cap Y_{0}$ is empty or parametrised by the elements of
the coset $G_{Y_{0}}/G_{y_{0}}$ with $y_{0}\in\mathcal{O}\cap Y_{0}$ since
$G_{y_{0}}$ is then a subgroup of $G_{Y_{0}}$. In particular, if the action of
$G$ on $X$ is assumed to be transitive, we have a unique orbit, $\mathcal{C}%
(X)=X$ and the atoms form a partition of $X$. Let us summarize the previous observations.

\begin{prop}
\label{Prop_LMF}Assume that $G$ acts on the set $X$ and $f:\mathcal{P}%
_{\mathrm{fin}}(X)\rightarrow\mathbb{R}$ is a $G$-invariant submodular
function such that $m=\inf_{Y\neq\emptyset\in\mathcal{P}_{\mathrm{fin}}%
(X)}f(Y)$ exists.

When the action of $G$ on $X$ is transitive, $\mathcal{C}(X)=X$, each element
of $X$ belongs to one atom.

Moreover the atoms for $f$ are blocks of imprimitivity of the action meaning
that we have the following properties:

\begin{enumerate}
\item The group $G$ acts on the set $\mathcal{A}$ of atoms for $f$.

\item The action of $G$ restricts to the core $\mathcal{C}(X)$ of $X$, defined
as the disjoint union of the atoms for $f$ which is thus also a disjoint union
of orbits for the action of $G$ on $X$.

\item For any atom $Y_{0}$, any $y_{0} \in Y_{0}$ and any orbit $\mathcal{O}$,
we have $G_{y_{0}} \subset G_{Y_{0}}$. Moreover the intersection set
$\mathcal{O}\cap Y_{0}$ is empty or parametrised by the elements of the coset
$G_{Y_{0}}/G_{y_{0}}$ with $y_{0}\in\mathcal{O}\cap Y_{0}$.
\end{enumerate}
\end{prop}

\begin{exa}
\label{ex-cut} For each action of a finite group $G$ on the finite set $X$,
one can consider the cut function $f$ as defined in (\ref{cut_function}). By
Proposition \ref{Prop_Submod_cut}, it is nonnegative submodular and 
$G$-invariant. Also the minimum of $f$ is equal to zero and is attained in any
subset $Y$ such that $g\cdot y\in Y$ for any $g\in G$ and any $y\in Y$. This
means that the fragments of $f$ are the disjoint union of orbits and the atoms
are the orbits of minimal cardinality. The core is the disjoint union of the
orbits with minimal cardinality.
\end{exa}

\begin{exa}
\label{ex-symgroup} Here is another example in which atoms are the orbits with
minimal cardinality; the submodular function considered is the function
$d_{A}$ of Subsection~\ref{sssec-dA} with $\lambda>0$. Fix $\sigma
\in\mathfrak{S}_{n}$, consider $X=\{1,\ldots,n\}$ and $A=\langle\sigma\rangle
$. As suggested in Section~\ref{subsection-orbit}, $X$ can be written as
$X=X_{1}\sqcup\cdots\sqcup X_{r}$ where the $X_{i}$ are the orbits of $X$
under the action of $A$. In this case, $d_{A}(Y)=\sum_{j,X_{j}\cap
Y\neq\emptyset}|X_{j}|-\lambda|Y|$. Among the subsets $Z$ of $X$ meeting non
trivially exactly the same $X_{i}$ as $Y$, $d_{A}(Z)$ is minimal precisely
when $Z=\cup_{j,X_{j}\cap Y\neq\emptyset}X_{j}$. In this case, $d_{A}%
(Z)=(1-\lambda)|Z|$. Thus, for $\lambda<1$, the fragments and atoms coincide
and are the $X_{j}$ with minimal cardinality. When $\lambda=1$, every union of
orbits is a fragment and the atoms are the $X_{j}$ with minimal cardinality.
\end{exa}

\section{Generalising results in additive group theory with submodular
functions}

\label{sec-application}

\label{Section_General_Hami} This section is devoted to the study of the
submodular functions $c_{Y},\gamma_{Y},d_{A}$ of \S ~\ref{sssec-cy-gammay} and
\S ~\ref{sssec-dA}. Each of its subsection is devoted to the study of one of
these three submodular maps. We start with $c_{Y}$ which allows us to
generalise three classical results. Subsection~\ref{ssec-gammay} is devoted to
$\gamma_{Y}$: we show that the results proved in the preceding subsection
extend to linear actions. Finally, in Subsection~\ref{ssec-dA}, we are able to
state results analogous to the one obtain for $c_{Y}$ in the case of an action
by an Abelian group. We also study the atoms for small or big values of the
parameter $\lambda$ in the definition of $d_{A}$.

In any cases, recall that we consider an action of the group $G$ on a set $X$
(or a linear action on a vector space $V$). The functions $c_{Y},\gamma_{Y}$,
$Y\subset X$ are defined on $\mathcal{P}_{\mathrm{fin}}(G)$ from a fixed
finite subset of $X$ or $V$ whereas the functions $d_{A}$ is defined on
$\mathcal{P}_{\mathrm{fin}}(X)$ from a finite fixed subset $A\subset G$. Also,
all these functions attain their minimum on their restrictions to nonempty
subsets as soon as they are nonnegative because their images are discrete
subsets of $\mathbb{R}$.

\subsection{Group action context and submodular functions $c_{Y}$}

\label{ssec-cy}

The submodularity and $G$-invariance of $c_{Y}$ allow us to generalise a
theorem of Hamidoune, a theorem of Petridis and Tao and a theorem of Tao on
small doubling sets.

\subsubsection{A generalisation of a theorem of Hamidoune}

Let us start with an observation which is not relevant in the context of
additive group theory but crucial in our group action context. Consider the
map%
\[
q_{Y}:\left\{
\begin{array}
[c]{l}%
\mathcal{P}_{\mathrm{fin}}(G)\setminus\{\emptyset\}\rightarrow\mathbb{Q}%
_{>0}\\
A\longmapsto\frac{\left\vert A\cdot Y\right\vert }{\left\vert A\right\vert }%
\end{array}
\right.
\]
Then it might happen that
\begin{equation}
\mu=\inf_{A\in\mathcal{P}_{\mathrm{fin}}(G)\setminus\{\emptyset\}}%
q_{Y}(A)=0.\label{H}%
\end{equation}
This will be in particular the case if $G_{Y}$ is an infinite subgroup of $G$
since subsets $A$ in $G_{Y}$ may have arbitrary large cardinalities whereas
$\left\vert AY\right\vert =\left\vert Y\right\vert $ is then fixed. In the
opposite direction, we will always have $\mu>0$ when

\begin{enumerate}
\item there exists an element $y_{0}\in Y$ such that $G_{y_{0}}=\{1\}$ and
then $\mu\geq1$ (this is in particular true if we consider the action by left
translation of $G$ on itself),

\item or the group $G$ is finite and then $\mu\geq\frac{\left\vert
Y\right\vert }{\left\vert G\right\vert }$ because we always have $\left\vert
A\cdot Y\right\vert \geq\left\vert Y\right\vert $ and $Y\ $is fixed.
\end{enumerate}

To overcome this difficulty, we need in general the assumption
\begin{equation}
\mu=\inf_{A\in\mathcal{P}_{\mathrm{fin}}(G)\setminus\{\emptyset\}}%
\frac{\left\vert A\cdot Y\right\vert }{\left\vert A\right\vert }%
>0.\label{def(mu)}%
\end{equation}

\begin{exa}
Let us compute the value of $\mu$ for some actions.

\begin{enumerate}
\item When the action is free (for example in the case of the left translation
of $G$ on itself), we have $\left\vert A\cdot Y\right\vert \geq\left\vert
A\right\vert $ so that $\mu\geq1$.

\item For the action of the symmetric group $\mathfrak{S}_{n}$ on
$\{1,\ldots,n\}$, when $\left\vert A\cdot Y\right\vert =\ell$, we get with the
notation of Example \ref{Example_CounterKneser}%
\[
\inf_{A\neq\emptyset\in\mathcal{P}_{\mathrm{fin}}(\mathfrak{S}_{n}%
)\mid\left\vert A\cdot Y\right\vert =\ell}\frac{\left\vert A\cdot Y\right\vert
}{\left\vert A\right\vert }=\frac{\left\vert A_{0}\cdot Y\right\vert
}{\left\vert A_{0}\right\vert }=\frac{\ell}{\frac{\ell!}{(\ell-k)!}(n-k)!}%
\]
which is minimal for $\ell=n$ and then%
\[
\mu=\frac{n}{n!}=\frac{1}{(n-1)!}.
\]

\item Assume $G$ is finite and acts on itself by conjugation. If we consider
$Y$ a subset of $Z(G)$, the center of $G$, we get $A\cdot Y=Y$ for any subset
$A\subset G$. Then%
\[
\mu=\inf_{A\in\mathcal{P}(G)\setminus\{\emptyset\}}\frac{\left\vert A\cdot
Y\right\vert }{\left\vert A\right\vert }=\frac{\left\vert Y\right\vert
}{\left\vert G\right\vert }.
\]

\item We get similarly $\mu=\frac{\left\vert Y\right\vert }{\left\vert
G\right\vert }$ as soon as $Y$ is a set of fixed elements under the action of
$G$.
\end{enumerate}
\end{exa}

\begin{rem}
Assume $G$ is infinite and the infimum $\mu$ in (\ref{def(mu)}) is attained
for the subset $A_{0}\subset G$, that is $\mu=\frac{\left\vert A_{0}\cdot
Y\right\vert }{\left\vert A_{0}\right\vert }>0$. Since we have $G_{Y}\cdot
Y=Y$ for the stabilizer $G_{Y}$ of $Y$, the set $A_{0}$ is a disjoint union of
left $G_{Y}$-cosets. In particular, $G_{Y}$ is finite.
\end{rem}

Under the assumption $\mu>0$, for any $\lambda\in\lbrack0,\mu]$, the
$G$-invariant submodular function $c_{Y}$ defined on $\mathcal{P}%
_{\mathrm{fin}}(G)$ by $c_{Y}(A)=\left\vert A\cdot Y\right\vert -\lambda
\left\vert A\right\vert $ is non negative since
\[
c_{Y}(A)=\left\vert A\cdot Y\right\vert -\lambda\left\vert A\right\vert
\geq\left\vert A\cdot Y\right\vert -\lambda\left\vert A\right\vert \geq0\,.
\]
Observe that
\[
c_{Y}(A)\geq(\mu-\lambda)\left\vert A\right\vert .
\]
We get the following theorem.

\begin{thm}
\label{Th_Hami}Consider a subset $Y\subset X$ and set
\[
\mu=\inf_{A\in\mathcal{P}_{\mathrm{fin}}(G)\setminus\{\emptyset\}}%
\frac{\left\vert A\cdot Y\right\vert }{\left\vert A\right\vert }.
\]
Then

\begin{itemize}
\item either $\mu=0$,

\item or for any $\lambda\in\lbrack0,\mu]$, there exists a finite subgroup $H
$ of $G$ containing $G_{Y}$ such that%
\begin{equation}
c_{Y}(A) \geq\ c_{Y}(H) \geq\left\vert Y\right\vert -\lambda\left\vert
H\right\vert \label{Hami}%
\end{equation}
for any finite subset $A$ in $G$.\ 
\end{itemize}
\end{thm}

\begin{proof}
Assume $\mu>0$ and set as usual $m=\min_{A\neq\emptyset\in\mathcal{P}%
_{\mathrm{fin}}(G)}c_{Y}(A)$. The case $\lambda=0$ is trivial (take $H=G_{Y}$
and remark that $G_{Y}$ is finite by a previous remark). Consider $\lambda
\in\rbrack0,\mu]$ and $A_{0}\in\mathcal{P}_{\mathrm{fin}}(G)$ such that
$\mu=\frac{\left\vert A_{0}\cdot Y\right\vert }{\left\vert A_{0}\right\vert }%
$. Then, for any $A\in\mathcal{P}_{\mathrm{fin}}(G)$, we have
\[
c_{Y}(A)=\left\vert A\cdot Y\right\vert -\lambda\left\vert A\right\vert
\geq\left\vert A_{0}\cdot Y\right\vert -\lambda\left\vert A_{0}\right\vert
\geq(\mu-\lambda)\left\vert A_{0}\right\vert \geq0
\]
so that $c_{Y}$ is a nonnegative submodular function.\ By Proposition
\ref{Prop_H}, there thus exists a unique atom $H$ for $c_{Y}$ containing $1$
which is a subgroup of $G$. Assume there exists $g\in G_{Y}$ such $g\notin H$.
Then
\[
c_{Y}(H\cup\{g\})=\left\vert (H\cup\{g\})\cdot Y\right\vert -\lambda\left\vert
H\cup\{g\}\right\vert =c_{Y}(H)-\lambda<c_{Y}(H)
\]
and $H\cup\{g\}$ is nonempty. This contradicts the fact that $H$ is an atom.
Thus, we must have $G_{Y}\subset H$. Also since $H$ is an atom, we have for
any finite subset $A$ in $G$%
\[
c_{Y}(A)=\left\vert A\cdot Y\right\vert -\lambda\left\vert A\right\vert
\geq\left\vert H\cdot Y\right\vert -\lambda\left\vert H\right\vert =c_{Y}(H)
\]
Since $1 \in H$, we have $Y \subset H \cdot Y$ which gives $c_Y(H) \geq |Y| - \lambda |H|$.
\end{proof}

\begin{rem}
\ 

\begin{enumerate}
\item Observe that when $\mu=0$, then $\lambda=0$ and the inequality
(\ref{Hami}) still holds since it reduces to $\left\vert A\cdot Y\right\vert
\geq\left\vert Y\right\vert $.

\item When $Y$ contains an element with trivial stabilizer, we have $\mu\geq1
$ and the theorem generalises Hamidoune's one when $G$ acts on itself by left translation.

\item Note that we must have $H=\{1\}$ when $G$ is torsion free because $H$ is
a finite subgroup of $G$.
\end{enumerate}
\end{rem}

Consider a finite subset $Y$ in $X$ such that $\mu>0$.

\begin{cor}
For any $\lambda\in]0,\mu]$ and any finite subset $A_{0}$ in $G$ there exists
a subgroup $H$ of $G$ containing $G_{Y}$ such that
\[
\lambda\max_{A\subset G\mid A\cdot Y=A_{0}\cdot Y}\left\vert A\right\vert
+\left\vert Y\right\vert \leq\lambda\left\vert H\right\vert +\left\vert
A_{0}\cdot Y\right\vert .
\]

\end{cor}

\subsubsection{A generalisation of a theorem of Petridis and Tao}

In another direction, we can also get the following analogue of a theorem by
Tao and Petridis (see \cite[Theorem 4.1]{Tao}) in our group action context.

\begin{thm}
\label{Th_Tao}Consider $A$ a nonempty finite subset of $G$ and $Y$ a finite
subset of $X$. Assume that
\[
\left\vert A\cdot Y\right\vert \leq\alpha\left\vert A\right\vert
\]
with $\alpha\in\mathbb{R}_{\geq0}$. Then, there exists a nonempty subset $B$
in $A$ such that%
\[
\left\vert CB\cdot Y\right\vert \leq\alpha\left\vert CB\right\vert
\]
for any finite subset $C$ of $G$.
\end{thm}

\begin{proof}
Define the map $q_{A,Y}$ such that%
\[
q_{A,Y}:\left\{
\begin{array}
[c]{l}%
\mathcal{P}(A)\setminus\{\emptyset\}\rightarrow\mathbb{Q}_{>0}\\
C\longmapsto\frac{\left\vert C\cdot Y\right\vert }{\left\vert C\right\vert }%
\end{array}
\right.
\]
and set its minimum $\mu$ (which indeed exists since $\mathcal{P}(A)$ is finite).
Let $B\subset A$ such that $\mu=\frac{\left\vert B\cdot Y\right\vert
}{\left\vert B\right\vert }$. Now consider the function $c_{Y}$ defined
on $\mathcal{P}_{\mathrm{fin}}(G)$ by $c_{Y}(C)=\left\vert C\cdot Y\right\vert
-\mu\left\vert C\right\vert $. We have seen that he function $c_{Y}$ is
submodular and $G$-invariant. We also have here $c_{Y}(B)=0$ and for any
$C\subset A$ we get $c_{Y}(C)\geq c_{Y}(B)=0$. Nevertheless, $c_{Y}$ may not be
nonnegative on $\mathcal{P}_{\mathrm{fin}}(G)$ in general.\ For any
nonempty finite subset $S$ of $G$ and any $g\in G$, we can write%
\[
c_{Y}(B\cup g^{-1}S)+c_{Y}(B\cap g^{-1}S)\leq c_{Y}(B)+c_{Y}(g^{-1}S)=
c_{Y}(S)
\]
because $c_{Y}(B)=0$ and $c_{Y}(g^{-1}S)=c_{Y}(S)$. We also have $c_{Y}(B\cap
g^{-1}S)\geq0$ because $B\cap g^{-1}S\subset B\subset A$ which implies that
$c_{Y}(B\cup g^{-1}S)\leq c_{Y}(S)$ for any $g\in G$ and any $S\in
\mathcal{P}_{\mathrm{fin}}(G)$. By $G$-invariance, this gives
\begin{equation}
c_{Y}(gB\cup S)\leq c_{Y}(S) \label{ineqcY}%
\end{equation}
for any $g\in G$ and any $S\in\mathcal{P}_{\mathrm{fin}}(G)$.
Now, let us consider a subset $C$ of $G$ such that $C=\{g_{1},g_{2}%
,\ldots,g_{m}\}$ and $C^{\flat}=\{g_{1},g_{2},\ldots,g_{m-1}\}$.\ We get for
any $S^{\prime}\in\mathcal{P}_{\mathrm{fin}}(G)$%
\[
c_{Y}(CB\cup S^{\prime})=c_{Y}(g_{m}B\cup(C^{\flat}B\cup S^{\prime}))\leq
c_{Y}(C^{\flat}B\cup S^{\prime})
\]
by applying (\ref{ineqcY}) with $g=g_{m}$ and $S=C^{\flat}B\cup S^{\prime}$.
By an easy induction on $m$ we finally obtain%
\[
c_{Y}(CB\cup S^{\prime})\leq c_{Y}(S^{\prime})
\]
for any $S^{\prime}\in\mathcal{P}_{\mathrm{fin}}(G)$. In particular for
$S^{\prime}=\emptyset$, we get%
\[
c_{Y}(CB)\leq0\Longleftrightarrow\left\vert CB\cdot Y\right\vert
-\mu\left\vert CB\right\vert \leq0\Longleftrightarrow\left\vert CB\cdot
Y\right\vert \leq\mu\left\vert CB\right\vert
\]
since $c_{Y}(\emptyset)=0$. We conclude by observing that $\mu=\min_{C\subset
A,C\neq\emptyset}\frac{\left\vert C\cdot Y\right\vert }{\left\vert
C\right\vert }\leq\frac{\left\vert A\cdot Y\right\vert }{\left\vert
A\right\vert }\leq\alpha$.
\end{proof}

\subsubsection{A generalisation of a theorem of Tao}

We can also use Theorem \ref{Th_Hami} to generalise the previous results and
obtain the following theorem which is also a generalisation of~\cite[Theorem
1.2]{Tao}.

\begin{thm}
Consider a discrete group $G$ acting on $X$. Let $A,Y$ be nonempty finite
subsets respectively of $G$ and $X$ such that $\left\vert A\right\vert
\geq\left\vert Y\right\vert $.\ Assume that%
\[
\mu=\inf_{S\neq\emptyset\in\mathcal{P}_{\mathrm{fin}}(G)}\frac{\left\vert
S\cdot Y\right\vert }{\left\vert S\right\vert }>0\text{ and there exists }
\varepsilon> 0 \text{ such that } \left\vert A\cdot Y\right\vert
\leq(2-\varepsilon)\mu\left\vert Y\right\vert .
\]
Then, there exists a finite subgroup $H$ of $G$ such that $Y$ is contained in
the disjoint union $H\cdot Y$ of $H$-orbits with%
\[
\left\vert H\right\vert \leq(\frac{2}{\varepsilon}-1)\left\vert Y\right\vert
\text{ and }\left\vert H\cdot Y\right\vert \leq\mu(\frac{2}{\varepsilon
}-1)\left\vert Y\right\vert .
\]

\end{thm}

\begin{proof}
Set $\lambda=\mu(1-\frac{\varepsilon}{2})$. By definition of $\mu$, we must have
\begin{equation}
c_{Y}(S)=\left\vert S\cdot Y\right\vert -\lambda\left\vert S\right\vert
\geq(\mu-\lambda)\left\vert S\right\vert \geq\mu\frac{\varepsilon}%
{2}\left\vert S\right\vert \geq0 \label{C_Y(S)}%
\end{equation}
for any finite subset $S\subset G$. From the hypotheses
$\left\vert A\cdot Y\right\vert \leq(2-\varepsilon)\mu\left\vert Y\right\vert
$ and $\left\vert A\right\vert \geq\left\vert Y\right\vert $, we obtain%
\begin{equation}
c_{Y}(A)=\left\vert A\cdot Y\right\vert -\mu(1-\frac{\varepsilon}%
{2})\left\vert A\right\vert \leq(2-\varepsilon)\mu\left\vert Y\right\vert
-\mu(1-\frac{\varepsilon}{2})\left\vert Y\right\vert =\mu(1-\frac{\varepsilon
}{2})\left\vert Y\right\vert . \label{C_Y(A)}%
\end{equation}
Let $H$ be the unique atom for $c_{Y}$ containing $1$. By Theorem
\ref{Th_Hami}$,$ we know that $H$ is a finite subgroup of $G$ and
$c_{Y}(H)\leq c_{Y}(A)$. We must have by (\ref{C_Y(S)}) and (\ref{C_Y(A)})%
\[
\left\vert H\right\vert \leq\frac{2}{\varepsilon\mu}c_{Y}(H)\leq\frac
{2}{\varepsilon\mu}c_{Y}(A)\leq(\frac{2}{\varepsilon}-1)\left\vert
Y\right\vert
\]
as desired.
We also get%
\[
c_{Y}(H)=\left\vert H\cdot Y\right\vert -\mu(1-\frac{\varepsilon}%
{2})\left\vert H\right\vert \leq c_{Y}(A)\leq\mu(1-\frac{\varepsilon}%
{2})\left\vert Y\right\vert .
\]
Therefore%
\[
\left\vert H\cdot Y\right\vert \leq\mu(1-\frac{\varepsilon}{2})\left\vert
H\right\vert +\mu(1-\frac{\varepsilon}{2})\left\vert Y\right\vert .
\]
Since $\mu=\inf_{S\in\mathcal{P}_{\mathrm{fin}}(G)\setminus\{\emptyset\}}%
\frac{\left\vert S\cdot Y\right\vert }{\left\vert S\right\vert }$ and
$H\in\mathcal{P}_{\mathrm{fin}}(G)\setminus\{\emptyset\}$, we should have
$\mu\left\vert H\right\vert \leq\left\vert H\cdot Y\right\vert $ which gives%
\[
\left\vert H\cdot Y\right\vert \leq(1-\frac{\varepsilon}{2})\left\vert H\cdot
Y\right\vert +\mu(1-\frac{\varepsilon}{2})\left\vert Y\right\vert .
\]
By gathering the occurrences of $\left\vert H\cdot Y\right\vert $, we finally
obtain the announced upper bound for $\left\vert H\cdot Y\right\vert $%
\[
\left\vert H\cdot Y\right\vert \leq\mu(\frac{2}{\varepsilon}-1)\left\vert
Y\right\vert .
\]
\end{proof}

\subsection{Group representation context and the submodular functions
$\gamma_{Y}$}

\label{ssec-gammay}

If we consider a representation $(\rho,V)$ of $G$ and a finite-dimensional $k
$-subspace $W=\langle Y\rangle$ in $V$, we can get an analogue of Theorem
\ref{Th_Hami} and of Theorem~\ref{Th_Tao}. The proof relies on the same
arguments and is thus omitted here.

\begin{thm}
\label{Th_Hami RT}Consider a finite-dimensional subspace $Y\subset V$ and set
\[
\mu=\inf_{A\neq\emptyset\in\mathcal{P}_{\mathrm{fin}}(G)}\frac{\dim(A\cdot
Y)}{\left\vert A\right\vert }.
\]
Then

\begin{itemize}
\item either $\mu=0$

\item or for any $\lambda\in\lbrack0,\mu]$, there exists a finite subgroup $H
$ of $G$ containing $G_{Y}$ such that%
\[
\dim(A\cdot Y)\geq\lambda\left\vert A\right\vert +\dim(H\cdot Y)-\lambda
\left\vert H\right\vert \geq\lambda\left\vert A\right\vert +\dim
(Y)-\lambda\left\vert H\right\vert
\]
for any subset $A$ in $G$.\ 
\end{itemize}
\end{thm}

\begin{thm}
Consider $A$ a finite nonempty subset of $G$ and $Y$ a finite-dimensional
$k$-subspace of $V$. Assume that
\[
\dim\langle A\cdot Y\rangle\leq\alpha\left\vert A\right\vert
\]
with $\alpha\in\mathbb{R}_{\geq0}$. Then, there exists a nonempty subset $B$
in $A$ such that%
\[
\dim\langle CB\cdot Y\rangle\leq\alpha\left\vert CB\right\vert
\]
for any finite subset $C$ of $G$.
\end{thm}

\subsection{Group action context and submodular functions $d_{A}$}

\label{ssec-dA}

\subsubsection{A generalisation of a theorem of Petridis and Tao}

\label{sssec_dA}Recall that for any fixed nonempty finite subset $A$ in $G$
and any $\lambda\geq0$, the submodular function $d_{A}$ is defined on
$\mathcal{P}_{\mathrm{fin}}(X)$ by $d_{A}(Y)=\left\vert A\cdot Y\right\vert
-\lambda\left\vert Y\right\vert $. Observe that the function $d_{A}$ is not
left invariant in general as defined but this is nevertheless the case when
$G$ is Abelian (see~\S ~\ref{sssec-dA}). The function $d_{A}$ is not
nonnegative for any $\lambda\geq0$ but this becomes true when $\lambda
\in\lbrack0,1]$ because we have for any non empty subset $A\subset G$ and any
$Y\subset X$ the inequality $\left\vert A\cdot Y\right\vert \geq\left\vert
Y\right\vert \geq\lambda\left\vert Y\right\vert $.

We get the following theorem which generalises~\cite[Theorem 4.1]{Tao}. It has
to be compared with Theorem~\ref{Th_Tao}.

\begin{thm}
\label{Th_Taod}Assume $G$ is Abelian. Consider $A$ a non empty finite subset
of $G$ and $Y$ a non empty finite subset of $X$. Assume that
\[
\left\vert A\cdot Y\right\vert \leq\alpha\left\vert Y\right\vert
\]
with $\alpha\in\mathbb{R}_{\geq0}$. Then, there exists a nonempty subset $Z$
in $Y$ such that%
\[
\left\vert AC\cdot Z\right\vert \leq\alpha\left\vert C\cdot Z\right\vert
\]
for any finite subset $C$ of $G$.
\end{thm}

\begin{proof}
Define the map $q_{Y,A}$ such that%
\[
q_{Y,A}:\left\{
\begin{array}
[c]{l}%
\mathcal{P}_{\mathrm{fin}}(Y)\setminus\{\emptyset\}\rightarrow\mathbb{Q}%
_{>0}\\
S\longmapsto\frac{\left\vert A\cdot S\right\vert }{\left\vert S\right\vert }%
\end{array}
\right.
\]
and its minimum $\mu$. Let $Z\subset Y$ such that $\mu=\frac{\left\vert A\cdot
Z\right\vert }{\left\vert Z\right\vert }$. Now consider the function
$d_{A}$ defined on $\mathcal{P}_{\mathrm{fin}}(X)$ by $d_{A}(S)=\left\vert
A\cdot S\right\vert -\mu\left\vert S\right\vert $. The function $d_{A}$ is
submodular and $G$-invariant because $G$ is Abelian. We have $d_{A}(Z)=0$ and
for any $S\subset Y$ we get $d_{A}(S)\geq0$. As in the proof of Theorem~\ref{Th_Tao}, the function $d_{A}$
is not nonnegative on $\mathcal{P}_{\mathrm{fin}}(X)$ in general.\ For any
nonempty finite subset $S$ of $X$ and any $g\in G$, we can write%
\[
d_{A}(Z\cup g^{-1}S)+d_{A}(Z\cap g^{-1}S)\leq d_{A}(Z)+d_{A}(g^{-1}S)\leq
d_{A}(S)
\]
because $d_{A}(Z)=0$ and $d_{A}(g^{-1}S)=d_{A}(S)$. We also have $d_{A}(Z\cap
g^{-1}S)\geq0$ because $Z\cap g^{-1}S\subset Z\subset Y$ which implies that
$d_{A}(Z\cup g^{-1}S)\leq d_{A}(S)$ for any $g\in G$ and any $S\in
\mathcal{P}_{\mathrm{fin}}(X)$. By $G$-invariance, this gives
\begin{equation}
d_{A}(gZ\cup S)\leq d_{A}(S) \label{ineqcA}%
\end{equation}
for any $g\in G$ and any $S\in\mathcal{P}_{\mathrm{fin}}(X)$.
Now, let us consider a subset $C$ of $G$ such that $C=\{g_{1},g_{2}%
,\ldots,g_{m}\}$ and $C^{\flat}=\{g_{1},g_{2},\ldots,g_{m-1}\}$.\ We get for
any $S^{\prime}\in\mathcal{P}_{\mathrm{fin}}(X)$%
\[
d_{A}((C\cdot Z)\cup S^{\prime})=d_{A}((g_{m}\cdot Z)\cup((C^{\flat}\cdot
Z)\cup S^{\prime}))\leq d_{A}((C^{\flat}\cdot Z)\cup S^{\prime}))
\]
by applying (\ref{ineqcA}) with $g=g_{m}$ and $S=(C^{\flat}\cdot Z)\cup
S^{\prime}$. By induction on $m$ we finally obtain%
\[
d_{A}((C\cdot Z)\cup S^{\prime})\leq d_{A}(S^{\prime})
\]
for any $S^{\prime}\in\mathcal{P}_{\mathrm{fin}}(X)$. In particular for
$S^{\prime}=\emptyset$, we get since $d_{A}(\emptyset)=0$%
\[
d_{A}(C\cdot Z)\leq0\Longleftrightarrow\left\vert A\cdot(C\cdot Z)\right\vert
-\mu\left\vert C\cdot Z\right\vert \leq0\Longleftrightarrow\left\vert AC\cdot
Z\right\vert \leq\mu\left\vert C\cdot Z\right\vert .
\]
We conclude by observing that $\mu=\min_{S\subset Y,S\neq\emptyset}%
\frac{\left\vert A\cdot S\right\vert }{\left\vert S\right\vert }\leq
\frac{\left\vert A\cdot Y\right\vert }{|Y|}\leq\alpha$.
\end{proof}

\bigskip

Under the hypotheses Theorem~\ref{Th_Taod}, we get the following interesting corollary.

\begin{cor}
Assume $G$ is Abelian and $\left\vert A\cdot Y\right\vert \leq\alpha\left\vert
Y\right\vert $. Then, there exists a nonempty subset $Z$ in $Y$ such that for
any integer $n\geq1$ we have
\[
\left\vert A^{n}\cdot Z\right\vert \leq\alpha^{n}\left\vert Z\right\vert .
\]

\end{cor}

\begin{proof}
By applying Theorem~\ref{Th_Taod}, we get a subset $Z$ of $Y$ such that $\left\vert
AC\cdot Z\right\vert \leq\alpha\left\vert C\cdot Z\right\vert $ for any finite
subset $C$ of $G$. In particular, with $C=\{1\},$ this gives $\left\vert
A\cdot Z\right\vert \leq\alpha\left\vert Z\right\vert $, that is the corollary
for $n=1$. Consider an integer $n\geq2$ and assume by induction that we
have $\left\vert A^{n-1}\cdot Z\right\vert \leq\alpha^{n-1}\left\vert
Z\right\vert .$ We then get%
\[
\left\vert A^{n}\cdot Z\right\vert =\left\vert A\cdot A^{n-1}\cdot
Z\right\vert \leq\alpha\left\vert A^{n-1}\cdot Z\right\vert \leq\alpha
^{n}\left\vert Z\right\vert
\]
where the first inequality is obtained by applying Theorem~\ref{Th_Taod} with $C=A^{n-1}$
and the second one is the induction hypothesis.
\end{proof}

\subsubsection{Behavior of the atoms for the submodular function $d_{A}$}

The function $d_{A}$ defined on $\mathcal{P}_{\mathrm{fin}}(X)$ by
$d_{A}(Y)=\left\vert A\cdot Y\right\vert -\lambda\left\vert Y\right\vert $ is
submodular nonnegative for any $\lambda\in\lbrack0,1]$ and left invariant when
$G$ is Abelian (see \S \ \ref{sssec-dA}).\ In contrast to
Examples~\ref{ex-cut} and~\ref{ex-symgroup}, the corresponding atoms and cores
depend on $\lambda$ and on the definition of the action.\ Our goal in this
paragraph is to show that, roughly speaking, the cardinality of fragments is
bounded by $\left\vert A\right\vert $ for small values of $\lambda$ whereas
for values of $\lambda$ close to $1$ and when the action is free, the
cardinality of fragments become larger than $\left\vert A\right\vert $.

More precisely, we have the following result.

\begin{prop}
Let $G$ be a group acting on $X$ (we do not assume that $G$ is Abelian) and $A
\subset G$.

\begin{enumerate}
\item Assume that $\lambda<1/|A|$. Then every fragment $Y$ for $d_{A}$
verifies $|Y|\leq|A|$.

\item Assume that the action of $G$ on $X$ is free and $A\subset G$ is such
that $\left\vert X\right\vert \geq\left\vert A\right\vert $.

For every $\mu\leq1$ and $Y \subset X$ such that $|Y| < \mu|A|$, $Y$ is not a
fragment for $d_{A}$ for every $\lambda$ verifying
\[
0 \leq\frac{|X|-|A|}{|X|-\mu|A|}\leq\lambda\leq1
\]
In particular, when $\mu>1-\frac{1}{\left\vert A\right\vert }$, the fragments
are of cardinality at least $|A|$ for any function $d_{A}$ such that%
\[
\lambda\geq\frac{|X|-|A|}{|X|-\mu|A|}\geq\frac{|X|-|A|}{|X|-|A|+1}.
\]

\end{enumerate}
\end{prop}

\begin{proof} Assume that $\lambda<1/|A|$. If we assume $|Y|\geq|A|+1$, we get for
any $y\in Y$%
\[
|A\cdot Y|-\lambda|Y|\geq(1-\lambda)|Y|\geq(1-\lambda
)(|A|+1)=|A|-\lambda+1-\lambda\left\vert A\right\vert >|A|-\lambda\geq
d_{A}(\{y\})
\]
because $\left\vert A\cdot\{y\}\right\vert \leq\left\vert A\right\vert $. This
gives the contradiction $d_{A}(\{y\})<d_{A}(Y)$.
Let us now consider the situation of $2.$ The freeness of the action insures us that $|A\cdot Y|\geq \left\vert A\right\vert $.
Hence we get for the function $d_{A}$ corresponding to $\lambda$%
\[
d_{A}(Y)=|A\cdot Y|-\lambda|Y|\geq|A|-\lambda|Y|>|A|-\mu\lambda|A|.
\]
By observing that
\[
\lambda\geq\frac{|X|-|A|}{|X|-\mu|A|}\Longleftrightarrow|A|-\mu\lambda
|A|\geq(1-\lambda)|X|
\]
and $d_{A}(X)=(1-\lambda)|X|$, we get that $Y$ cannot be a fragment.
\end{proof}

Thus, atoms and fragments indeed strongly depend on $\lambda$ and are in
general not easy to determine explicitly.

%It follows that for any atom $W_{0}$ and any $g\in G$%
%\[
%g\cdot W_{0}\text{ is an atom such that }g\cdot W_{0}=W_{0}\text{ or }g\cdot
%W_{0}\cap W_{0}=\{0\}\text{.}%
%\]
%Let $\mathcal{A}_{k}$ be the set of atoms for $f$. We thus get an action of
%the group $G$ on the set of atoms $\mathcal{A}$. For any vector $v_{0}$ in the
%atom $W_{0}$, we yet obtain the inclusion $S_{y_{0}}\subset S_{Y}$ of the
%stabilizers. Moreover if $w_{0}$ belongs to the atom $W_{0}$, then any element
%$g\cdot W_{0},g\in G$ also belongs to a atom. We call the set
%\[
%\mathcal{C}_{k}(X)=%
%%TCIMACRO{\tbigoplus \limits_{W_{0}\in\mathcal{A}}}%
%%BeginExpansion
%{\textstyle\bigoplus\limits_{W_{0}\in\mathcal{A}}}
%%EndExpansion
%W_{0}%
%\]
%the $k$-\emph{core} of representation $V$.\ The action of $G$ on $V$ restricts
%to an action on $\mathcal{C}_{k}(V)$.\ When the underlying representation
%theory is semisimple (for example when $G$ is finite and $k=\mathbb{C}$), the
%space $\mathcal{C}_{k}(V)$ decomposes in a direct sum
%\[
%\mathcal{C}_{k}(V)=%
%%TCIMACRO{\tbigoplus \limits_{i\in I}}%
%%BeginExpansion
%{\textstyle\bigoplus\limits_{i\in I}}
%%EndExpansion
%V_{i}%
%\]
%of irreducible representations for $G$.
%Can we go further here ?

\section{Other generalisations}

\label{sec-generelisation}

We may define submodular functions on a lattice $(S,\vee,\wedge)$ by the
inequalities $f(a\wedge b)+f(a\vee b)\leq f(a)+f(b)$ for every $a,b\in S$. In
particular, one can consider the lattice $(\mathcal{P}_{k\mathrm{fin}%
}(V),+,\cap)$ of finite-dimensional vector subspaces of a linear
representation $V$ of $G$. If $f$ is a submodular function such that
$m=\inf_{\{0\}\neq Y\in\mathcal{P}_{k\mathrm{fin}}(V)}f(Y)$ exists. We define
a fragment for $f$ as a $k$-vector subspace $W\in\mathcal{P}_{k\mathrm{fin}%
}(V)$ which is not reduced to $\{0\}$ and such that $f(W)=m$ and an atom for
$f$ as a fragment with minimal dimension. All the atoms have the same
dimension and we have a linear analogue of Lemma \ref{Lem_EmptyInter} whose
proof is similar.

\begin{lem}
\label{Lem_EmptyInterLinear} If $W_{1}$ and $W_{2}$ are two atoms of $f$ then
$W_{1}=W_{2}$ or $W_{1}\cap W_{2}=\{0\}$.
\end{lem}

Fix $(\rho,V)$ a linear representation of an Abelian group $G$ and $A$ a non
empty subset of~$G$. The map $d_{A}$ of \S \ref{sssec-dA} can be adapted to a
$G$-invariant submodular map $\delta_{A}$
\[
\delta_{A}:\left\{
\begin{array}
[c]{l}%
\mathcal{P}_{\mathrm{kfin}}(V)\rightarrow\mathbb{R}\\
Y\longmapsto\left\vert A\cdot Y\right\vert -\lambda\dim Y
\end{array}
\right.
\]

Using $\gamma_{A}$ we obtain group representation versions of the results in
subsection~\ref{ssec-dA} when $(\rho,V)$ is a representation of $G$.

\begin{thm}
Assume $G$ is Abelian. Consider $A$ a non empty finite subset of $G$ and $Y$ a
finite-dimensional $k$-subspace of $V$. Assume that
\[
\dim(A\cdot Y)\leq\alpha\dim(Y)
\]
with $\alpha\in\mathbb{R}_{\geq0}$. Then, there exists a $k$-subspace
$Z\neq\{0\}$ in $Y$ such that%
\[
\dim(AC\cdot Z)\leq\alpha\dim(C\cdot Z)
\]
for any finite subset $C$ of $G$.
\end{thm}

\begin{cor}
Assume $G$ is Abelian and $\dim A\cdot Y \leq\alpha\dim Y$. Then, there exists
a $k$-subspace $Z\neq\{0\}$ in $\langle Y\rangle$ such that for any integer
$n\geq1$ we have
\[
\dim(A^{n}\cdot Z)\leq\alpha^{n}\dim Z.
\]

\end{cor}

\bigskip

%\section{Some extensions}
%\subsection{Monoid and semi-group actions}
%Case of monoid which can be embedded in a group.
%\subsection{Action of a group on another group}
%Action of a group on itself, action of $\mathbb{Z}$ on $g$ ($G$ regarded as a
%$G$ module), semidirect product of two groups
%\bigskip
%\bigskip
%\bigskip
%\bigskip

%\bigskip
%\noindent Vincent Beck
%\noindent Institut Denis Poisson (UMR CNRS 7013)\newline Universit\'{e} d'Orl\'{e}ans
%\bigskip
%\noindent C\'{e}dric Lecouvey
%\noindent Institut Denis Poisson (UMR CNRS 7013)\newline Universit\'{e} de Tours

\end{document}